\documentclass[reqno,a4paper,11pt]{amsart}
\usepackage{color}
\usepackage[latin1]{inputenc}
\usepackage[english]{babel}
\usepackage{amsmath,amssymb,amsfonts,amsthm,enumerate}
\usepackage{mathrsfs}
\usepackage[T1]{fontenc}
\usepackage[active]{srcltx}
\usepackage{epsfig}
\usepackage{graphicx}

\numberwithin{equation}{section} 
\newtheorem{theorem}{Theorem}[section]
\newtheorem{lemma}[theorem]{Lemma}

\newtheorem{remark}[theorem]{Remark}
\newtheorem{proposition}[theorem]{Proposition}

\newcommand{\N}{\mathbb{N}}
\newcommand{\R}{\mathbb{R}}

\newcommand{\Lin}{\mathscr{L}}

\newcommand{\PP}{\mathscr{P}}

\renewcommand{\>}{\rangle}
\renewcommand{\a}{\alpha}

\newcommand{\e}{\varepsilon}
\newcommand{\g}{\gamma}

\renewcommand{\l}{\lambda}

\renewcommand{\S}{\mathbb{S}}

\newcommand{\p}{\partial}

\newcommand{\Id}{\operatorname{Id}}

 \newcommand{\tauV}{{\kern-3pt\tau}}

 \newcommand{\oVVVk}{\overline{\mbox{\boldmath$V$}}\kern-3pt}
 \newcommand{\tVVVk}{\tilde{\mbox{\boldmath$V$}}\kern-3pt}

 \newcommand{\negint}{{\int\negthickspace\negthickspace\negthickspace\negthickspace-}}

\begin{document}

\title[Free boundary problems for fully nonlinear elliptic equations]{A general class of free boundary problems\\ for fully nonlinear elliptic equations}

\author[Alessio Figalli]{Alessio Figalli}
\address{Mathematics Department, The University of Texas at Austin,  Austin, Texas, 78712-1202, USA}
\email{figalli@math.utexas.edu}

\author[Henrik Shahgholian ]{Henrik Shahgholian}
\address{Department of Mathematics, KTH Royal Institute of Technology, 100~44  Stockholm, Sweden}
\email{henriksh@kth.se}

\thanks{A. Figalli was partially supported by NSF grant DMS-0969962.
H. Shahgholian was partially supported by Swedish Research Council. 
A. Figalli acknowledges the G\"oran Gustafsson Foundation for his visiting appointment to KTH}

\begin{abstract}    
In this paper we study the fully nonlinear free boundary problem 
$$
\left\{
\begin{array}{ll}
F(D^2u)=1 & \text{a.e. in }B_1 \cap \Omega\\
|D^2 u| \leq K & \text{a.e. in }B_1\setminus\Omega ,
\end{array}
\right.
$$
where $K>0$, and $\Omega$ is an unknown open set.

Our main result is the optimal regularity for solutions to this problem:
namely, we prove that $W^{2,n}$ solutions are locally $C^{1,1}$ inside $B_1$.
Under the extra condition that $\Omega \supset \{D u\neq 0 \}$, and  a uniform thickness assumption on the coincidence set $\{D u =  0 \}$,
we also show local regularity for the free boundary $\partial\Omega\cap B_1$.
\end{abstract}

\maketitle


\section{Introduction and main result}

\subsection{Background}
Since the seminal work of Luis A. Caffarelli \cite{Caff} on the analysis of free boundaries in the obstacle problem, many new techniques and tools have been developed to treat similar type of free boundary problems. The linear theory, i.e., when the operator  is the Laplacian, has been completely resolved in \cite{CKS,Sh} for Lipschitz right hand side $f$ and 
when the equation is satisfied outside the set where $u$ vanishes (this correspond to the obstacle problem):
\begin{equation}
\Delta u=f\chi_{\{u\neq 0\}} \quad \text{in }B_1 .
\end{equation}

Passing below the Lipschitz threshold was a challenging task, as the previous techniques were using monotonicity formulas which failed when $f \in C^\alpha$. The main difficulty has been to prove the $C^{1,1}$-regularity of solutions. On the other hand the regularity of the free boundary for the Laplacian case was still feasible (even in low-regularity cases) due to the fact  that after blow-up the right hand side becomes a constant, and hence the monotonicity tool applies again. (We refer to the above reference for more details.) 

A generalization of the problem towards fully nonlinear operator $F(D^2 u)=\chi_{\{u\neq 0\}}$ 
 for the signed-problem (i.e., $u\geq 0$) was completely done by K. Lee \cite{Lee} and later 
partial results were obtained by Lee-Shahgholian in the case of no-sign obstacle problem \cite{LS}. Here, two challenging problems were left: 
(i) $C^{1,1}$-regularity of $u$; (ii) Classification of global solutions.

Recently, using harmonic analysis technique,  Andersson-Lindgren-Shahgholian \cite{ALS} could prove a complete result for the Laplacian case, with $f$ satisfying a Dini-condition. Actually their argument shows that if the elliptic equation $\Delta v=f$ admits a $C^{1,1}$-solution in $B_1$, then the corresponding free boundary problem also admits a $C^{1,1}$-solution. From here, the free boundary regularity follows as in the classical case. The heart of the matter in \cite{ALS} lies in their Proposition 1 (due to John Andersson) which is a dichotomy between the growth of the solution and the decay of the volume of the coincidence set. Indeed, one can show that if (close to a free boundary point)  the growth of the solution  is not quadratic,
then the volume of the complement set $B_r(x^0)\setminus \Omega$ decays fast enough to make the potential of this set twice differentiable at the origin. From this fact, they can then achieve the optimal growth.

In \cite{ALS} the authors strongly relied on the linearity of the equation to consider projections
of the solution onto the space of second order harmonic polynomials. 
Also, the linearity of the equation plays a crucial role in several of their estimates.
Here, we introduce a suitable ``fully nonlinear version'' of this projection
operation, and we are able to circumvent the difficulties coming from the nonlinear structure of the problem to prove  $C^{1,1}$ regularity of the solution. 
Using this result, we can also show $C^1$-regularity of the free boundary under a uniform thickness assumptions on the ``coincidence set'', which proves in particular that Lipschitz free boundaries are smooth.
Nevertheless, a complete regularity of the free boundary  still remains open due to  lack of new technique to classify global solutions.

\subsection{Setting of the problem}
Our aim here is to provide an optimal regularity result for solutions to a very general class of free boundary problems which include both the obstacle problem (i.e., the right hand side is given
by $\chi_{\{u\neq 0\}}$) and the more general free boundary problems studied in \cite{CS}
(where the right hand side is of the form $\chi_{\{\nabla u \neq 0\}}$).

To include these examples in a unique general framework,
we make the weakest possible assumption on the structure of the equation:
we suppose that $u$ solves a fully nonlinear equation
inside an open set $\Omega$, and in the complement of $\Omega$ we only assume that 
$D^2u$ is bounded.

Notice that, in the above mentioned problems, the first step in the regularity theory is to show that viscosity solutions are $W^{2,p}$ for any $p<\infty$ (this is a relatively ``soft'' part), and then one wants to prove
that actually solutions are $C^{1,1}$.

Since the first step is already pretty well understood \cite{F,CS,PSU}, 
here we focus on the second one. Hence, we assume  that 
$u:B_1 \to \R$ is a $W^{2,n}$ function satisfying 
\begin{equation}
\label{eq:obstacle}
\left\{
\begin{array}{ll}
F(D^2u)=1 & \text{a.e. in }B_1 \cap \Omega\\
|D^2 u| \leq K & \text{a.e. in }B_1\setminus\Omega ,
\end{array}
\right.
\end{equation}
where $K>0$, and $\Omega \subset \R^n$ is some unknown open set.
Since $D^2u$ is bounded in the complement of $\Omega$, we see that $F(D^2u)$
is bounded inside the whole $B_1$, therefore $u$ is a so-called ``strong $L^n$ solution''
to a fully nonlinear equation with bounded right hand side \cite{CCKS}.
We refer to \cite{CC} as a basic reference to fully nonlinear equations and viscosity methods, and to \cite{F,CS,PSU} for several existence results for strong solutions to free boundary type problems.

Let us observe that, if $u \in W^{2,n}$, then $D^2u=0$ a.e. inside both sets $\{u=0\}$ and $\{\nabla u=0\}$,
so \eqref{eq:obstacle} includes as special cases both $F(D^2u)=\chi_{\{u \neq 0\}}$
and $F(D^2u)=\chi_{\{\nabla u \neq 0\}}$.

We assume that:
\begin{enumerate}
\item[(H0)] $F(0)=0$.
\item[(H1)] $F$ is uniformly elliptic with ellipticity constants $\l_0,\lambda_1>0$, that is
$$
\PP^-(Q-P)
 \leq F(Q) - F(P) \leq \PP^+(Q-P)
$$
for any $P,Q$ symmetric, where $\PP^-$ and $\PP^+$ are the extremal Pucci operators:
$$
\PP^-(M):=\inf_{\lambda_0 \Id \leq N \leq \lambda_1 \Id} {\rm trace}(NM),\qquad \PP^+(M):=\sup_{\l_0 \Id \leq N \leq \l_1 \Id} {\rm trace}(NM).
$$
\item[(H2)] $F$ is either convex or concave. 
\end{enumerate}

Under assumptions (H0)-(H2) above, strong $L^n$ solutions are also viscosity solutions \cite{CCKS},
so classical regularity results for fully nonlinear equations \cite{CaFNL}
show that $u \in W^{2,p}_{\rm loc}(B_1)$ for all $p <\infty$.
In addition, by \cite{CH}, $D^2u$  belongs to BMO.

Our primary aim here is to prove uniform optimal $C^{1,1}$-regularity for $u$.
This is a key step in order to be able to perform an analysis of the  free boundary.

\begin{remark}
In order to keep the presentation simple and to highlight the main ideas in the proof,
 we decided to restrict ourselves to the ``clean'' case $F(D^2u)= 1$ inside $\Omega$. However, under suitable regularity assumptions on $F$ and $f$, we expect our arguments  to work for the general class of equations
    $F(x,u,\nabla u, D^2 u)=f$ inside $\Omega$.
\end{remark}

\subsection{Main results}
Our main result in this paper concerns optimal regularity of solutions to \eqref{eq:obstacle}.

\begin{theorem} (Interior $C^{1,1}$ regularity)
\label{thm:C11}
Let $u:B_1 \to \R$ be a $W^{2,n}$ solution of \eqref{eq:obstacle}, and
assume that $F$ satisfies (H0)-(H2).
Then there exists a universal constant $\bar C>0$ such that
$$
|D^2u| \leq \bar C, \qquad \hbox{in } B_{1/2}.
$$
\end{theorem}

In order to investigate the regularity of the free boundary,
we need to restrict ourselves to a more specific situation than the one in \eqref{eq:obstacle}.
Indeed, as discussed in Section~\ref{sect:non deg},
even if we assume that $D^2u=0$ outside $\Omega$,
non-degeneracy of solutions (a key ingredient to study the regularity of the free boundary)
may fail.
As we will see, a sufficient condition to show non-degeneracy of solutions
is to assume that $\Omega\supset \{\nabla u\neq 0\}$.
Still, once non-degeneracy is proved, the 
lack of strong tools (available in the Laplacian case) such as monotonicity formulas
makes the regularity of the free boundary a very challenging issue.

To state our result we need to introduce the concept of minimal diameter.
Set $\Lambda:=B_1\setminus \Omega$, and for any set $E$ let
$\operatorname{MD}(E)$ denote  the smallest possible distance between two parallel hyperplanes containing $E$. Then,
we define the minimal diameter of the set $\Lambda$ in $B_r(x)$ as
$$
\delta_r(u,x):=\frac{\operatorname{MD}(\Lambda \cap B_r(x))}{r}.
$$
We notice that $\delta_r$ enjoys the scaling property
$\delta_1(u_r,0)= \delta_r(u,x)$, where $u_r(y)=u(x+ry)/r^2$. 

Our result provides regularity for the free boundary under a uniform thickness condition.
As a corollary of our result, we deduce that Lipschitz free boundaries
are $C^1$, and hence smooth \cite{KN}.

\begin{theorem} (Free boundary regularity)\label{thm:C12}
Let $u:B_1 \to \R$ be a $W^{2,n}$ solution of \eqref{eq:obstacle}. 
Assume that $F$ is convex and satisfies (H0)-(H1), and that 
one of the following conditions holds:
\begin{enumerate}
\item[-]  $\Omega\supset\{\nabla u \neq 0\}$ and $F$ is of class $C^1$;
\item[-] $\Omega\supset\{u \neq 0\}$.
\end{enumerate}
Suppose further that there exists $\e>0$ such that
$$
\delta_r(u,z) > \e
\qquad \forall \,r<1/4,\,z\in \partial \Omega \cap B_{r}(0).
$$ 
Then 
$\partial \Omega \cap B_{r_0}(0)$  is  a $C^1$-graph, where $r_0$
depends only on $\e$ and the data.
\end{theorem}

The important difference between this theorem and previous results of this form
is that here we assume thickness of $\Lambda$  in a  uniform neighborhood of the origin rather than  at 
the origin only.
The reason for this fact is that this allows us to classify global solutions arising as
blow-ups around such ``thick points''. Once this is done, then local regularity follows in pretty standard way.

The paper is organized as follows:
In Section~\ref{regularity} we prove Theorem~\ref{thm:C11}.
Then in Section~\ref{sect:non deg} we investigate the non-degeneracy of solutions, and classify global solutions under a suitable thickness assumption. In  Section~\ref{dir-mon} we show directional monotonicity for local solutions, that gives a Lipschitz regularity for the free boundary.
This Lipschitz regularity can then be improved to $C^1$.
The details of such an analysis are by-now classical and only indicated shortly in Section~\ref{Local-regularity}.


\section{Proof of Theorem~\ref{thm:C11}}\label{regularity}

\subsection{Technical preliminaries}\label{technicality}

In this section we shall gather some technical tools that are interesting in their own rights, and may even be applied to
other problems. Throughout all the section, we assume that  $F$ satisfies (H0)-(H2).

With no loss of generality, here we will perform all our estimates at the origin, and later on we will apply such estimates at all points where $u$ is twice differentiable, showing that $D^2u$
is universally bounded at all such points. This will give a complete optimal regularity for $u$; see Section~\ref{main-proof}.

For all $r<1/4$, we define
\begin{equation}
\label{eq:Ar}
A_r:=\{x: \ rx \in B_r\setminus \Omega\}=
\frac{B_r \setminus \Omega}{r}\subset B_1.
\end{equation}

We recall that, by \cite[Theorem A]{CH},
$$
\|D^2u\|_{BMO(B_{3/4})} \leq C,
$$
which implies in particular that
\begin{equation}
\label{eq:BMO 0}
\sup_{r\in (0,1/4)} \negint_{B_r(0)} |D^2 u(y) -(D^2u)_{r,0}|^2\,dy \leq C,\qquad (D^2u)_{r,0}:=\negint_{B_r(0)} D^2 u(y)\,dy.
\end{equation}

Here we first show that in \eqref{eq:BMO 0} we can replace $(D^2u)_{r,0}$
with a matrix in $F^{-1}(1)$.

\begin{lemma}
\label{lemma:Pr}
There exists $C>0$ universal such that
\begin{equation}
\label{eq:BMO 1}
\min_{F(P)=1}\negint_{B_r(0)} |D^2 u(y) -P|^2\,dy \leq C \qquad \forall\,r \in (0,1/4).
\end{equation}
\end{lemma}

\begin{proof}
Set $Q_r:=(D^2u)_{r,0}$. Since
$F(D^2u)$ is bounded inside $B_1$ and
$F$ is $\lambda_1$-Lipschitz (this is a consequence of (H1)),
using \eqref{eq:BMO 0} we get
\begin{align*}
|F(Q_r)|&=\left|\negint_{B_r(0)}F\left(Q_r-D^2u(y)+D^2u(y)\right) \,dy \right|\\
&\leq \negint_{B_r(0)}\left(\left|F(D^2u(y))\right| +\lambda_1 \left|Q_r-D^2u(y)\right| \right)\,dy\\
&\leq 1+ \sqrt{\negint_{B_r(0)} |D^2 u(y) -(D^2u)_{r,0}|^2\,dy} \leq C.
\end{align*}
Thus we have proved that $F(Q_r) $ is universally bounded.
By ellipticity and continuity (see (H1)) we easily deduce that there exists a universally bounded
constant $\beta \in \R$ such that
$F(Q_r+\beta \Id)=1$.
Since
$$
\negint_{B_r(0)} |D^2 u(y) -(Q_r+\beta \Id)|^2\,dy \leq 2\negint_{B_r(0)} |D^2 u(y) -Q_r|^2\,dy + 2 \beta^2,
$$
this proves the result.
\end{proof}

For any $r \in (0,1/4)$,
let $P_r \in F^{-1}(1)$ denote a minimizer in \eqref{eq:BMO 1} (although $P_r$ may not be unique,
we just choose one).

We first show that $P_r$ cannot change too much on a dyadic scale:
\begin{lemma}
\label{lem:dyadic}
There exists a universal constant $C_0$ such that
$$
|P_{2r}-P_r| \leq C_0 \qquad \forall\,r \in (0,1/8).
$$ 
\end{lemma}
\begin{proof}
By the estimate
$$
\negint_{B_r(0)} |D^2 u(y) -P_r|^2\,dy+\negint_{B_{2r}(0)} |D^2 u(y) -P_{2r}|^2\,dy \leq C
$$
(see \eqref{eq:BMO 1}), we obtain
\begin{align*}
|P_{2r} - P_r|^2 &\leq 2
\negint_{B_r(0)} |D^2 u(y) -P_r|^2\,dy+2\negint_{B_{r}(0)} |D^2 u(y) -P_{2r}|^2\,dy\\
&\leq 2
\negint_{B_r(0)} |D^2 u(y) -P_r|^2\,dy+2^{n+1}\negint_{B_{2r}(0)} |D^2 u(y) -P_{2r}|^2\,dy \leq C,
\end{align*}
which proves the result.
\end{proof}

The following result shows that if $P_r$ is bounded, then (up to a linear function) so is $|u|/r^2$ inside $B_r$.
\begin{lemma}
\label{lem:bounded}
Assume that $u(0)=\nabla u(0)=0$. Then
there exists a universal constant $C_1$ such that
\begin{equation}
\label{eq:uPr}
\sup_{B_r(0)}\left| u - \frac{1}{2}\<P_r y,y\>\right| \leq C_1r^2  \qquad \forall\,r \in (0,1/8).
\end{equation}
In particular
\begin{equation}
\label{eq:ur2}
\sup_{B_r(0)}|u| \leq (C_1+|P_r|) r^2.
\end{equation}
\end{lemma}
\begin{proof}
By Lemma~\ref{lemma:Pr} we know that
$$
\left\| D^2\frac{u(r y)}{r^2} - P_r\right\|_{L^2(B_1)} \leq C,
$$
that is the function $\bar u_r(y):=u(ry)/r^2 - \frac{1}{2}\<P_r y,y\>$ satisfies
$$
\left\| D^2\bar u_r \right\|_{L^2(B_1)} \leq C.
$$
We see that $F(P_r+ D^2 \bar u_r(y))=F(D^2u(ry)) \in L^\infty(B_1)$, $F(P_r)=1$, and $\bar u_r(0)=\nabla \bar u_r(0)=0$. Hence,
by interior $C^{1,\alpha}$ estimates for the elliptic operator $G(Q):=F(P_r+Q)-1$
(see for instance \cite[Chapter 5.3]{CC} and \cite[Theorem 2]{CaFNL}), we deduce that
$$
\|\bar u_r\|_{C^{1,\alpha}(B_{1/2})} \leq C.
$$
In particular
\begin{equation}
\label{eq:bound infty}
\sup_{B_{r/2}(0)}\left| \frac{u - \frac{1}{2}\<P_r y,y\>}{r^2}\right|=\sup_{B_{1/2}}|\bar u_r| \leq C.
\end{equation}
To prove that actually we can replace $r/2$ with $r$ in the equation above (see \eqref{eq:uPr}), we first apply \eqref{eq:bound infty} with $2r$ in place of $r$ to get
$$
\sup_{B_{r}(0)}\left| \frac{u - \frac{1}{2}\<P_{2r} y,y\> }{(2r)^2}  \right|\leq C, 
$$
and then we conclude by Lemma~\ref{lem:dyadic}.
\end{proof}

We now prove that if $|P_r|$ is sufficiently large then the measure of $A_r$ (see \eqref{eq:Ar})
has to decay in a geometric fashion.

\begin{proposition}
\label{prop:M}
There exists $M>0$ universal such that, for any $r \in (0,1/8)$, if $|P_r| \geq M$ then
$$
|A_{r/2}| \leq \frac{|A_r|}{2^n}.
$$
\end{proposition}

\begin{proof}
Set $u_r(y):=u(ry)/r^2$, and let
\begin{equation}\label{rewrite}
u_r(y)=\frac{1}{2}\<P_r y,y\> + v_r(y)+w_r(y),
\end{equation}
where $v_r$ is defined as the solution of
\begin{equation}
\label{eq:vr}
\left\{
\begin{array}{ll}
F(P_r+D^2v_r)-1=0 & \text{in }B_1,\\
v_r=u_r(y)-\frac{1}{2}\<P_r y,y\>& \text{on }\p B_1,
\end{array}
\right.
\end{equation}
and by definition $w_r:=u_r-\frac{1}{2}\<P_r y,y\> -v_r$.

Set $f_r:=F(D^2u_r) \in L^\infty(B_1)$ (recall that $|D^2 u_r|\leq K$ a.e. inside $A_r$, see \eqref{eq:obstacle}).
Notice that, since $f_r=1$ outside $A_r$,
$$
F(D^2u_r)-F(P_r+D^2v_r) =(f_r-1)\chi_{A_r},
$$
so it follows by (H1) that $w_r$ solves
\begin{equation}
\label{eq:wr}
\left\{
\begin{array}{ll}
\PP^-(D^2w_r) \leq (f-1)\chi_{A_r} \leq \PP^+(D^2w_r) & \text{in }B_1,\\
w_r=0& \text{on }\p B_1.
\end{array}
\right.
\end{equation}
Hence, since  $f_r$ is universally bounded, we can apply the ABP estimate \cite[Chapter 3]{CC} to deduce that
\begin{equation}
\label{eq:ABPwr}
\sup_{B_1}|w_r| \leq C \| \chi_{A_r}\|_{L^n(B_1(0))} = C|A_r|^{1/n}.
\end{equation}
Also, since $F(P_r)=1$ and $v_r$ is universally bounded on $\partial B_1$ (see \eqref{eq:uPr}), by Evans-Krylov's theorem \cite[Chapter 6]{CC} applied to \eqref{eq:vr} we have
\begin{equation}
\label{eq:vrsmooth}
\|D^2 v_r\|_{C^{0,\alpha}(B_{3/4}(0))} \leq C. 
\end{equation}
This implies that $w_r$ solves the fully nonlinear equation with H\"older coefficients
$$
G(x,D^2w_r)=(f_r-1)\chi_{A_r} \quad \text{in }B_{3/4},\qquad G(x,Q):=F(P_r+D^2 v_r(x) +Q)-1.
$$
Since $G(x,0)=0$, we can apply \cite[Theorem 1]{CaFNL} with $p=2n$,
and using \eqref{eq:ABPwr} we obtain
\begin{equation}
\label{eq:wrAr}
\int_{B_{1/2}(0)}|D^2 w_r|^{2n} \leq
C\left(\|w_r\|_{L^\infty(B_{3/4})} + \| \chi_{A_r}\|_{L^{2n}(B_{3/4}(0))} \right)^{2n} \leq C\, |A_r|
\end{equation}
(recall that $|A_r|\leq |B_1|$).

We are now ready to conclude the proof:
since $|D^2 u_r|\leq K$ a.e. inside $A_r$ (by \eqref{eq:obstacle}),
recalling (\ref{rewrite}) we have
$$
\int_{A_r\cap B_{1/2}(0)}|D^2 v_r+D^2 w_r + P_r|^{2n}
=\int_{A_r \cap B_{1/2}(0)} |D^2 u_r|^{2n} \leq K^{2n}|A_r|.
$$
Therefore, by \eqref{eq:vrsmooth} and \eqref{eq:wrAr},
\begin{align*}
|A_r\cap B_{1/2}(0)|\,|P_r|^{2n}&=\int_{A_r\cap B_{1/2}(0)}|P_r|^{2n}\\
&\leq 3^{2n}\biggl( \int_{A_r\cap B_{1/2}(0)}|D^2 v_r|^{2n} +  \int_{A_r\cap B_{1/2}(0)}|D^2 w_r|^{2n}
+K^{2n}|A_r|\biggr)\\
&\leq  3^{2n} \biggl(|A_r\cap B_{1/2}(0)|\,\|D^2 v_r\|_{L^\infty(B_{1/2}(0))}+ \int_{B_{1/2}(0)}|D^2 w_r|^{2n}+K^{2n}|A_r|\biggr)\\
&\leq C\,|A_r\cap B_{1/2}(0)| + C\, |A_r|.
\end{align*}
Hence, if $|P_r|$ is sufficiently large we obtain
$$
|A_r\cap B_{1/2}(0)|\,|P_r|^{2n} \leq C|A_r| \leq  \frac{1}{4^n}|P_r|^{2n}|A_r|.
$$
Since $|A_{r/2}|= 2^n|A_r\cap B_{1/2}(0)|$, this gives the desired result.
\end{proof}


\subsection{Proof of Theorem~\ref{thm:C11}}\label{main-proof}

Since by assumption $|D^2u|\leq K$ a.e. outside $\Omega$, it suffices to prove that
$|D^2u(x^0)|\leq C$ for a.e. $x^0 \in \bar\Omega \cap B_{1/2}$, for some $C>0$ universal.

Fix $x^0 \in \bar\Omega \cap B_{1/2}$ such that 
$u$ is twice differentiable at $x^0$, and $x^0$ a Lebesgue point for $D^2 u$ (these properties
hold at almost every point).
With no loss of generality we can assume that $x^0=0$ and that $u(0)=\nabla u(0)=0$.

Let $M >0$ as in Proposition~\ref{prop:M}.
We distinguish two cases:

\begin{enumerate}
    \item[(i)]  $\liminf_{k\to \infty}|P_{2^{-k}}| \leq 3M$. 
    \item[(ii)]   $\liminf_{k\to 0}|P_{2^{-k}}| \geq 3M$. 
\end{enumerate}
    
   Using \eqref{eq:ur2} and the fact that $u$ is twice differentiable at $0$, in case  (i)
   we immediately obtain
$$
|D^2u(0)|\leq \liminf_{k\to \infty} \sup_{B_{2^{-k}}(0)}\frac{2|u|}{2^{-2k}} \leq 2(C_1+3M).
$$

In case (ii), let us define
$$
k_0:=\inf\Bigl\{ k \geq 2\,:\, |P_{2^{-j}}| \geq 2M \quad \forall\,j \geq k\Bigr\}.
$$
By the assumption that $\liminf_{k\to 0}|P_{2^{-k}}| \geq 3M$, we see that $k_0<\infty$.
In addition, since $P_{1/4}$ is universally bounded, up to enlarge $M$
we can assume that $k_0\geq 3$.

Let us observe that, since by definition $|P_{2^{-k_0-1}}|\leq 2M$, by Lemma~\ref{lem:dyadic} we obtain 
\begin{equation}
\label{eq:Pk0}
|P_{2^{-k_0}}| \leq 2M+C_0.
\end{equation}

We now define the function $\bar u_0:=4^{k^0}u(2^{-k_0}x) - \frac{1}{2}\<P_{2^{-k_0}}x,x\>$.
Observe that $\bar u_0$ is a solution of the fully nonlinear equation
\begin{equation}
\label{eq:u0}
G(D^2 \bar u_0)=(f_{2^{-k_0}}-1)\chi_{A_{2^{-k_0}}}\qquad \text{in }B_1,
\end{equation}
where $G(Q):=F(P_{2^{-k_0}}+Q)-1$ and $f_{2^{-k_0}}(x):=F(D^2u(2^{-k_0}x))$ is universally bounded.
In addition, since $|P_{2^{-k}}| \geq 2M$ for all $k \geq k_0$,
Proposition~\ref{prop:M} gives
$$
|A_{2^{-k_0+j}}| \leq 2^{-jn}|A_{2^{-k_0}}|\leq 2^{-jn} \qquad \forall\, j \geq 0,
$$
from which we deduce  that $(f_{2^{-k_0}}-1)\chi_{A_{2^{-k_0}}}$ decays in $L^n$ geometrically fast:
$$
\negint_{B_r}\bigl|(f_{2^{-k_0}}-1)\chi_{A_{2^{-k_0}}}\bigr|^n \leq C\negint_{B_r}|\chi_{A_{2^{-k_0}}}| \leq Cr^n \qquad \forall\, r\in (0,1).
$$
Hence, since $G(0)=0$, we can apply \cite[Theorem 3]{CaFNL} to deduce that $\bar u_0$ is $C^{2,\alpha}$ at the origin, with universal bounds. In particular this implies
$$
|D^2\bar u_0(0)| \leq C.
$$
Since $D^2u(0)=D^2\bar u_0(0)+P_{2^{-k_0}}$ and $P_{2^{-k_0}}$ is universally bounded
(see \eqref{eq:Pk0}), this concludes the proof.


\section{Non-degeneracy and global solutions}
\label{sect:non deg}

\subsection{Local non-degeneracy}
Non-degeneracy is a corner-stone for proving smoothness of the free boundary. 
This property says that the function grows quadratically (and not slower) away from the free boundary points, that is, $\sup_{B_r(x^0)} |u - u(x^0) - (x-x^0)\cdot \nabla u (x^0)| \gtrsim  r^2$ for any $x^0 \in \overline\Omega$.
However, while in the case $\Delta u=\chi_{\{u \neq 0\}}$ or $\Delta u=\chi_{\{\nabla u \neq 0\}}$ non-degeneracy is known to hold true,
in the case $\Delta u=\chi_{\{D^2u\neq 0\}}$ non-degeneracy may fail.

To see this, one can consider the one dimensional problem $u''=\chi_{\{u''\neq 0\}}$.
Every solution is obtained by linear functions and quadratic polynomial glued together in a $C^{1,1}$ way. 
In particular, if $\{I_j\}_{j\ in \N}$ is a countable family of disjoint intervals, the function
$$
u(t):=\int_0^t \int_0^s \chi_{\Omega}(\tau)\,d\tau\,ds,\qquad \Omega:=\cup_jI_j
$$
satisfies $u''=\chi_{\Omega}=\chi_{\{u''\neq0\}}$, and if we choose $I_j$ such that
$$
\frac{\left|\Omega\cap (-r,r)\right|}{2r} \to 0 \qquad \text{as $r \to 0$},
$$
then it is easy to check that $u(r)=o(r^2)$ as $r \to 0$.

A possible way to rule out the above counterexample may be to consider only points in $\overline\Omega$ such that $\Omega$ has a uniform density inside $B_r(x^0)$.
We will not investigate this direction here. Instead, we show that non-degeneracy 
holds under the additional assumption that $\Omega\supset\{\nabla u \neq 0\}$
(which is sufficient to include into our analysis the cases $F(D^2 u)=\chi_{\{u \neq 0\}}$ and
$F(D^2u)=\chi_{\{\nabla u \neq 0\}}$).

\begin{lemma}\label{supercond-nondegeneracy}
Let $u:B_1 \to \R$ be a $W^{2,n}$ solution of \eqref{eq:obstacle}, 
assume that $F$ satisfies (H0)-(H2), and that $\Omega\supset\{\nabla u \neq 0\}$.
Then, for any $x^0 \in \overline\Omega\cap B_{1/2}$,
$$
\max_{\p B_r(x^0)} u \geq u(x^0)+\frac{r^2}{2n\lambda_1}\qquad \forall\, r \in (0,1/4).
$$
\end{lemma}

\begin{proof}
By approximation, it suffices to prove the estimate for $x^0 \in \Omega$.
In addition, since $D^2u=0$ a.e. inside the set $\{\nabla u=0\}$, $F(D^2u)=1$ in $\Omega$,
and $F(0)=0$ (by (H0)), we see that $\{\nabla u=0\}$ has measure zero inside $\Omega$.
This implies that the set $ \Omega\cap \{\nabla u\neq 0\}$ is dense inside $\overline\Omega$,
and so we only need to prove the result when $x^0 \in \Omega\cap \{\nabla u\neq 0\}$.

Let us define the $C^{1,1}$ function (recall that $u \in C^{1,1}$ because of Theorem~\ref{thm:C11})
$$
v(x):=u(x)-\frac{|x-x^0|^2}{2n\lambda_1}.
$$
By (H1) we see that
\begin{equation}
\label{eq:vsub}
F(D^2v)=F\bigl(D^2u - \Id/(n\lambda_1)\bigr)
\geq F(D^2u)-\PP^+\bigl(\Id/(n\lambda_1)\bigr) \geq 0 \qquad \text{in }\Omega \cap B_1.
\end{equation}
We claim that 
$$
\max_{\p B_r(x^0)}v =\sup_{B_r(x^0)}v.
$$
Indeed, if there exists an interior maximum point $y \in B_r(x^0)$,
then 
\begin{equation}
\label{eq:nabla}
0=\nabla v(y)=\nabla u (y) - \frac{y-x^0}{n\l_1}.
\end{equation}
Since by assumption $x^0 \in \{\nabla u\neq 0\}$ we have $\nabla u(x^0)\neq 0$, so $y \neq x^0$.
In particular $\nabla u (y)= \frac{y-x^0}{n\l_1} \neq 0$, and thus $y \in \Omega$.
Recalling that $v$ is a subsolution for $F$ inside $\Omega \cap B_1$ (see (H0) and \eqref{eq:vsub}),
by the strong maximum principle $v$ is constant inside the connected component of $\Omega\cap  B_r(x^0)$ containing $y$. Hence $\nabla v(y)=0$ inside such component, so by \eqref{eq:nabla}
we deduce that  this connected component contains the  whole ball $B_r(x^0)$ (since $\nabla u=0$ outside $\Omega$).
Hence $v$ is constant in $B_r(x^0)$ and claim is trivially true.

Thanks to the claim we obtain
$$
\max_{\p B_r(x^0)} u -\frac{r^2}{2n\l_1}=\max_{\p B_r(x^0)}v \geq v(x^0)=u(x^0),
$$
which proves the result.
\end{proof}

\subsection{Classification of global solutions}

Now that non-degeneracy is proven, we can start considering blow-up
solutions and try to classify them. We shall treat the case 
$\Omega\supset\{\nabla u \neq 0\}$. Our results would work also for the case $\Omega \supset \{D^2 u \neq 0\}$ if the assumptions are strengthened in a way 
that solutions stay stable/invariant in a blow-up regime. 
 
Since we will use  minimal diameter to measure sets, we need some facts about the  stability property minimal diameter.
Let us first recall the definition for $\delta_r(u,x)$:
$$
\delta_r(u,x):=\frac{\operatorname{MD}(\Lambda \cap B_r(x))}{r},\qquad \Lambda :=B_1\setminus \Omega.
$$
We remark that, for polynomial global solutions $P_2=\sum a_j\,x_j^2$ (with $a_j$ such that $F(D^2P_2)=1$), one has 
\begin{equation}\label{P2}
\delta_r(P_2,0)=0.
 \end{equation}
Indeed, the zeros of the  gradient of  a second degree homogeneous polynomial $P_2$  always lie on a hyperplane.

The next observation is the stability of $\delta_r(u,x)$ under scaling: more precisely,
if $x \in \partial\Omega\cap B_1$ and we rescale $u$ as $u_r(y):=\frac{u(x+ry) - u(x)}{r^2}$
 (notice that $\nabla u(x)=0$ for all $x \in \partial \Omega$), then
\begin{equation}\label{stability1}
\delta_r(u,x)= \delta_1(u_r,0)
 \end{equation}
 which along with the fact that $\limsup_{r\to 0}\Lambda (u_r) \subset \Lambda (u_0)$ whenever $u_r $ converges to some function $u_0$ (see \cite[Proposition 3.17 (iv)]{PSU}) gives 
 \begin{equation}\label{stability2}
\limsup_{r \to 0} \delta_r(u,x^0)\leq  \delta_1(u_0,0) .
  \end{equation} 
Since any limit of $u_r$ will be a global solution of \eqref{eq:obstacle}
(i.e., it solves \eqref{eq:obstacle} in the whole $\R^n$), we are interested in classifying global solutions.

In the next proposition we classify global solution with a ``thick free boundary''.

\begin{proposition}\label{global-convexity}
Let $u:\R^n \to \R$ be a $W^{2,n}$ solution of \eqref{eq:obstacle} inside $\R^n$, 
assume that $F$ is convex and satisfies (H0)-(H1), and that $\Omega\supset\{\nabla u \neq 0\}$.
Assume that there exists $\epsilon_0>0$ such that
\begin{equation}\label{min-diam}
\delta_r(u,x^0) \geq \epsilon_0 \qquad \forall\,r>0,\,\forall\,x^0 \in \partial\Omega.
\end{equation}
Then $u$ is a half-space solution, i.e., up to a rotation,
$u(x)=\g [(x_1)_+]^2/2+c$, where $\g \in (1/\lambda_1,1/\l_0) $ is such that $F(\g e_1\otimes e_1 )=1$ and $c \in \R$.
\end{proposition}

\begin{proof}

We first prove that $u$ is convex.
Suppose by contradiction that  $u$ is not,
and set
    $$
    m:=\inf_{z \in \Omega,\,e \in \mathbb\S^{n-1}} \partial_{ee}u(z)<0.
    $$
    Observe that, thanks to Theorem~\ref{thm:C11}, $u$ is globally $C^{1,1}$ in $\R^n$, so $m$ is finite.
   
     Let us consider sequences $y^j\in \Omega$ and $e^j \in \mathbb\S^{n-1}$ such that
        $$
         \partial_{e^je^j}u(y^j) \to m\qquad \text{as $j\to\infty$}.
        $$
        Rescale $u$ at $y^j$ with respect to $d_j:={\rm dist}(y^j,\partial\Omega)$, i.e.,
        $$u_j(x):=\frac{u(d_jx+y^j)-u(y^j) -d_j \nabla u(y^j)\cdot x}{d_j^2}.
        $$
        Also, up to rotate the system of coordinates,
        we can assume that (up to subsequences) $e^j \to e_1$.
        Then the functions $u^j$ still satisfy \eqref{eq:obstacle}, and they converge to another global solution $u_\infty$ which satisfies
        $\partial_{11}u_\infty(0)=-m$.
        Let us observe that, by convexity of $F$, $\partial_{11}u_\infty$ is a supersolution of the linear operator
        $F_{ij}(D^2u_\infty)\partial_{ij}$.
         Hence,
        since $\partial_{11}u_\infty(z) \geq -m$ inside $B_1(0)$,
        by the strong maximum principle we deduce that 
        $\partial_{11}u_\infty$ is constant inside the connected component containing $B_1(0)$
        (call it $\Omega_\infty$).
    Also, since $D_{ee}u_\infty(z) \geq -m$ inside $B_1(0)$
        for any $e \in \mathbb\S^{n-1}$, it follows that $e_1$ is an eigenvector of $D^2u$ at every point
        (which corresponds to the smallest eigenvalue).
        In particular this implies that $\partial_{1j}u_\infty=0$ for any $j=2,\ldots,n$ inside $\Omega$.
    Hence, integrating $u_\infty$
        in the direction $e_1$ gives 
    \begin{equation} 
    \label{eq:formula u0}
     u_\infty(x)= P(x) \qquad \text{inside $\Omega_\infty$,}
      \end{equation}
      where 
      $$
      P(x):=-mx_1^2/2 + ax_1+b(x'), \quad x'=(x_2, \ldots, x_n).
       $$
      We now observe that the set where $\partial_1P$ vanishes corresponds to the hyperplane $\{x_1=a/m\}$.
      Since $\nabla u_\infty=0$ on $\partial \Omega_\infty$, we deduce that
      $\partial \Omega_\infty \subset \{x_1=a/m\}$.
       We now distinguish two cases:
\begin{enumerate}
\item[-]
    If $\partial \Omega_\infty \neq \{x_1=a/m\}$ then the set $\Omega_\infty$ contains $\R^n\setminus \{x_1=a/m\}$ (since $\nabla u_\infty$
      could not vanish anywhere else), and so $F(D^2 u_\infty)=1$ a.e. in $\R^n$.
      Since $u_\infty$ grows at most quadratically (because of the global $C^{1,1}$ regularity), we can apply Evans-Krylov's theorem \cite[Chapter 6]{CC} to $u_\infty(Ry)/R^2$ inside $B_1$ to deduce that
      $$
      \sup_{x,z \in B_{R}} \frac{|D^2u_\infty(x)-D^2u_\infty(z)|}{|x-z|^\alpha} \leq \frac{C}{R^\alpha}.
      $$
      Letting $R \to \infty$ we obtain that $D^2u_\infty$ is constant, and so $u_\infty$ is a second
      order polynomial.
\item[-]
    If $\partial \Omega_\infty= \{x_1=a/m\}$, since $\nabla u_\infty=0$ on $\partial \Omega_\infty$
    we get that $\nabla_{x'}P=0$ on the hyperplane $ \{x_1=a/m\}$.
    Hence $b$ is constant and so
    $$
    u_\infty= -mx_1^2/2 + ax_1+b\qquad \text{inside $\{x_1>a/m\}$},
    $$
     which contradicts (H0) and (H1) (because $F(D^2u_\infty)=1$
    while $D^2u_\infty=-m{\rm Id}$ is negative definite).
   \end{enumerate}
    In conclusion we have proved that if $u$ is not convex, then $u_\infty$ is a second order polynomial.
    Invoking the minimal diameter assumption (\ref{min-diam}) and the stability 
    properties (\ref{stability1})-(\ref{stability2}) along with (\ref{P2}), we conclude
    that $u_\infty$ cannot be a second degree polynomial, and thus a contradiction.

 Hence, we have proved that $u$ is convex.
 Recall that, since $F(D^2u)=1$ in $\Omega$, we have $|\Omega \setminus \{\nabla u=0\}|$,
 and by convexity of $u$ and the thickness assumption it is easy to see that $\Omega=\{\nabla u\neq 0\}$.
 
 We now show that the set $\Lambda (u)=\{\nabla u=0\}$ is a half-space.
For simplicity we may assume the origin is on the free boundary. 
Consider a blow-down $u_\infty$ obtained as a limit (up to a subsequence) of $u(Ry)/R^2$ as $R \to \infty.$
It is not hard to realize that $\Lambda (u_\infty)= \{x \in \Lambda (u): \ tx \in \Lambda (u) \ \forall \,t>0  \}$. In other words, the  coincidence set
 for the blow-down is convex,
 and coincides with the largest cone (with vertex at the origin) in the coincidence set of the function $u$.
 Assume by contradiction that $\Lambda (u_\infty)$ is not a half-space.
 Then, in some suitable system of coordinates
 $$
 \Lambda(u_\infty) \subset \mathcal C_{\theta_0}:=\bigl\{x \in \R^n : \ x=(\rho \cos\theta,\rho\sin\theta,x_3,\ldots,x_n), \, \theta_0\leq |\theta|\leq \pi\bigr\}
 $$
 for some $\theta_0>\pi/2$. Hence, if we choose $\theta_1\in (\pi/2,\theta_0)$
and set $\alpha:=\pi/\theta_1$, 
then it is easy to check that, for $\beta>0$ sufficiently large (the largeness depending only on $\theta_1$
and the ellipticity constants of $F$), the function 
$$
v=r^\a\bigl(e^{-\beta\sin(\a\theta)} - e^{-\beta}\bigr)
$$
is a positive subsolution for the linear operator $F_{ij}(D^2u)\partial_{ij}$
inside $\R^n\setminus \mathcal C_1$ (see for instance \cite{Lee}), and it vanishes on
 $\partial\mathcal C_{\theta_1}$.
 Hence, because $\partial_1 u_\infty>0$ inside $\R^n\setminus \mathcal C_{\theta_0}$ (by convexity of $u_\infty$) and $\theta_0>\theta_1$,
by the comparison principle we deduce that 
$$
v \leq \partial_1 u_\infty.
$$
However, since $\a<1$, this contradicts the Lipschitz regularity of $\partial_1 u_\infty$ at the origin.

So $\Lambda (u_\infty)$ is a half space, and since $\Lambda (u_\infty) \subset \Lambda (u)$
and the latter set is convex,
we deduce   that $\Lambda (u)$ is a half-space as well.
 
 Finally, to conclude the proof, it suffices to consider the function $w$ obtained by reflecting
 $u$ with respect to $\partial \Lambda (u)$ in a even fashion: in this way $w \in C^{1,1}$
 (since $\nabla u$ vanishes on $\partial \Lambda(u)$)
 and $F(D^2w)=1$
 a.e. in $\R^n$, so $w$ is a second order polynomial. Hence also $u$ is a second order polynomial
 inside the half-space $\R^n\setminus \Lambda (u)$,
 and since $\nabla u=0$ on the hyperplane $\partial \Lambda(u)$, it is immediate  to check that
 it has to be a half-space solution.
  \end{proof}


\section{Local solutions and directional monotonicity}\label{dir-mon}

In this section we shall prove a directional monotonicity for solutions to our equations. 
In the next section we will use Lemmas~\ref{lemma:monotonicity 1} and~\ref{lemma:monotonicity 2}
below to show that,
if $u$ is close enough to a half-space solution $\g [(x_1)_+]^2$ in a ball $B_r$, then
for any $e \in \mathbb S^{n-1}$ with $e \cdot e_1\geq s>0$ we have $C_0\partial_e u  - u \geq 0 $ inside
$B_{r/2}$.

\subsection{The case $\Omega\supset \{u\neq 0\}$}

\begin{lemma}
\label{lemma:monotonicity 1}
Let $u:B_1 \to \R$ be a $W^{2,n}$ solution of \eqref{eq:obstacle}
with $\Omega\supset\{u \neq 0\}$.
Assume that $C_0\partial_eu-u \geq -\e_0$ in $B_1$ for some $C_0,\e_0 \geq 0$,
and that $F$ is convex and satisfies (H0)-(H1).
Then $C_0\partial_eu-u \geq 0$ in $B_{1/2}$ provided $\e_0 \leq 1/(8n\l_1)$.
\end{lemma}
\begin{proof}
Since $F$ is convex, for any matrix $M$ we can choose an element $P^M$ inside $\partial F(M)$
(the subdifferential of  $F$ at $M$)
in such a way that the map $M \mapsto P^M$ is measurable.
Then, since that $u \in C^{2,\alpha}_{\rm loc}(\Omega)$ (by Evans-Krylov's Theorem \cite[Chapter 6]{CC}), we can define the measurable uniformly elliptic coefficients
$$
a_{ij}(x):=(P^{D^2u(x)})_{ij} \in \partial F(D^2u(x)).
$$
We now notice two useful facts:
first of all, since $a_{ij} \in \partial F(D^2u)$,
by convexity of $F$ we deduce that, for any $x \in \Omega$ and $h >0$ small such that $x+he \in \Omega$,
$$
a_{ij}(x) \frac{\partial_{ij}u(x+he)-\partial_{ij}u(x)}{h} \leq \frac{F(D^2u(x+he)) - F(D^2u(x))}{h}=0,
$$
so, by letting $h \to 0$,
\begin{equation}
\label{est:1}
a_{ij}\partial_{ij} \partial_eu \leq 0\qquad \text{in }\Omega.
\end{equation}
Also, again by the convexity of $F$ and recalling that $F(0)=0$, we have
\begin{equation}
\label{est:2}
a_{ij}\partial_{ij}u \geq F(D^2u) - F(0)=1\qquad \text{in }\Omega.
\end{equation}
Now, let us assume by contradiction that there exists $y_0\in B_{1/2}$ such that $C_0\partial_eu(y_0)-u(y_0)<0$, and consider
the function
$$
w(x):=C_0\partial_eu(x)- u(x)+\frac{|x-y_0|^2}{2n\lambda_1}.
$$
Thanks to \eqref{est:1}, \eqref{est:2}, and assumption (H1) (which implies that $\lambda_0 \Id \leq a_{ij}\leq \lambda_1 \Id$) we deduce that
$w$ is a supersolution of the linear operator $\Lin:=a_{ij}\partial_{ij}$.
Hence, by the maximum principle,
$$
\min_{\partial (\Omega \cap B_1)} w =  \min_{\Omega \cap B_1} w \leq w(y_0)<0,
$$
where the first inequality follows from the fact that $y_0 \in \Omega \cap B_{1/2}$ (since $u=\nabla u=0$ outside $\Omega$).

Since $w\geq 0$ on $\partial \Omega$ and $|x-y_0|^2\geq 1/4$ on $\partial B_1$, it follows that
$$
0>\min_{\partial B_1} w \geq -\e_0+ \frac{1}{8n\lambda_1},
$$
a contradiction if $\e_0 < 1/(8n\lambda_1)$.
\end{proof}

\subsection{The case $\Omega\supset \{\nabla u \neq 0\}$}

\begin{lemma}
\label{lemma:monotonicity 2}
Let $u:B_1 \to \R$ be a $W^{2,n}$ solution of \eqref{eq:obstacle}
with $\Omega\supset\{\nabla u \neq 0\}$.
Assume that $C_0\partial_eu-|\nabla u|^2 \geq -\e_0$ in $B_1$ for some $C_0,\e_0 \geq 0$,
and that $F$ is convex, of class $C^1$, and satisfies (H0)-(H1).
Then $C_0\partial_eu-|\nabla u|^2  \geq 0$ in $B_{1/2}$ provided $\e_0 \leq \l_0/(4n^2\lambda_1^3)$.
\end{lemma}
\begin{proof}
By differentiating the equation $F(D^2u)=1$ inside $\Omega$, we deduce that 
\begin{equation}
\label{est:1bis}
F_{ij}(D^2u)\partial_{ij}\nabla u =0.
\end{equation}
We now observe that, since $F_{ij} \in C^0$ (because $F \in C^1$) and $D^2u \in C^{2,\alpha}_{\rm loc}(\Omega)$
(by Evans-Krylov's Theorem \cite[Chapter 6]{CC}), $\nabla u$ solves
a linear elliptic equation with continuous coefficients, so by standard elliptic theory $\nabla u \in W^{2,p}_{\rm loc}(\Omega)$ for any $p<\infty$.
Hence, we can apply the linear operator $F_{ij}(D^2u)\partial_{ij}$ to the $W^{2,p}_{\rm loc}$ function $|\nabla u|^2$, and using \eqref{est:1bis} we obtain 
\begin{align*}
F_{ij}(D^2u)\partial_{ij} |\nabla u|^2& = 2\left(F_{ij}(D^2u)\partial_{ij}\partial_k u\right) \cdot \partial_k u + 2F_{ij}(D^2u) \partial_{ij}u\partial_{ik}u\\
&=2F_{ij}(D^2u) \partial_{ij}u\partial_{ik}u.
\end{align*}
Now, if for every point $x \in \Omega$ we choose a system of coordinates so that $D^2u$ is diagonal, since $F_{ii}(D^2u)\geq \lambda_0$ for all $i=1,\ldots,n$ (by (H1)) we obtain
$$
F_{ij}(D^2u(x))\partial_{ij} |\nabla u|^2(x) = 2F_{ii}(D^2u(x))\left(D_{ii}u(x)\right)^2 \geq 2\lambda_0 |D^2u(x)|^2.
$$
Using (H1) again, we also have
$$
1 = F(D^2 u) - F(0) \leq \sqrt{n}\l_1|D^2u|\qquad \text{inside $\Omega$},
$$
so by combining the two estimates above we condude that
\begin{equation}
\label{est:2bis}
F_{ij}(D^2u))\partial_{ij} |\nabla u|^2 \geq 2\lambda_0/(n\lambda_1^2).
\end{equation}
Thanks to \eqref{est:1bis} and \eqref{est:2bis}, we conclude exactly as before considering now the function
$$
w(x):=C_0\partial_eu(x)- |\nabla u|^2 (x)+\frac{\l_0|x-y_0|^2}{n^2\lambda_1^3}.
$$
\end{proof}


\section{Proof of Theorem~\ref{thm:C12}}\label{Local-regularity}
As already mentioned in the introduction,
once we know that blow-up solutions around ``thick points''
are half-space solutions (Proposition~\ref{global-convexity})
and we can improve almost directional monotonicity to full directional monotonicity
(Lemmas ~\ref{lemma:monotonicity 1} and~\ref{lemma:monotonicity 2}), then the proof of Theorem~\ref{thm:C12} becomes standard.
For convenience of the reader, we briefly sketch it here.

We consider only the case when $\Omega \supset \{u \neq 0\}$ (the other being analogous).

Take $x \in \partial \Omega \cap B_{1/8}$,
and rescale the solution around $x$, that is, consider $u_r(y):=[u(x+ry)-u(x)]/r^2$.
Because of the uniform $C^{1,1}$ estimate provided by Theorem~\ref{thm:C11},
we can find a sequence $r_j\to 0$
such that $u_{r_j}$ converges  locally in $C^1$ to a 
global solution $u_0$ satisfying 
$u_0(0)= 0$.
Moreover, by our thickness assumption on the free boundary of $u$ and  \eqref{stability2}, it follows that
the minimal diameter property holds for all $r>0$ and all points on the free boundary $\partial \Omega (u_0)$. Then, by Proposition~\ref{global-convexity} we deduce that $u_0$ is of the form
$u_0=\g[(x\cdot e_x)_+]^2/2 $ with $\gamma \in [1/\l_1,1/\l_0]$ and $e_x \in \mathbb S^{n-1}$.

Notice now that, for any $s>0$, we can find a large constant $C_s$ such that
$$
C_s\p_e u_ 0- u_0 \geq 0,\qquad C_s\p_e u_ 0- |\nabla u_0|^2 \geq 0\qquad \text{inside $B_1$}
$$
for all directions $e \in\mathbb S^{n-1}$ such that $e\cdot e_x \geq s$.
Since $u_{r_j} \to u_0$ in $C^1_{\rm loc}$, we deduce 
that, for $j$ sufficiently large (the largeness depending on $s$),
the assumptions of Lemma~\ref{lemma:monotonicity 1} are satisfied with $u=\tilde u_{r_j}$.
Hence
\begin{equation}
\label{eq:monot}
C_s\partial_e  u_{r_j}-  u_{r_j} \geq 0\qquad \text{in $B_{1/2}$},
\end{equation}
and since $u_{r_j}(0)=0$ a simple ODE argument 
shows that $u_{r_j} \geq 0$ in $B_{1/4}$  (see the proof of Lemmas 4.4-4.5 in \cite{PSU}).

Using \eqref{eq:monot} again, this implies that 
$\partial_e  u_{r_j}$ inside $B_{1/4}$,
and so in terms of $u$ we deduce that there exists $r=r(s)>0$ such that
$$
\partial_e u \geq 0 \qquad \text{inside $B_{r}(x)$}
$$
for all $e \in \mathbb S^{n-1}$ such that $e \cdot e_x \geq s$.

A simple compactness argument shows that $r$ is independent of the point $x$,
which implies that the free boundary is $s$-Lipschitz.
Since $s$ can be taken arbitrarily small (provided one reduces the size of $r$), this actually
proves that the free boundary is $C^1$ (compare for instance with \cite[Theorem 4.10]{PSU}).
Higher regularity follows from the classical work of Kinderlehrer-Nirenberg \cite{KN}.


\end{document}